\documentclass[12pt,reqno]{amsart}

\usepackage{times}
\usepackage[T1]{fontenc}
\usepackage{mathrsfs}
\usepackage{latexsym}
\usepackage[dvips]{graphics}
\usepackage{epsfig}
\usepackage{verbatim}

\usepackage{amsmath,amsfonts,amsthm,amssymb,amscd}
\input amssym.def
\input amssym.tex
\usepackage{color}
\usepackage{hyperref}

\usepackage{color}
\usepackage{breakurl}
\usepackage{comment}
\newcommand{\bburl}[1]{\textcolor{blue}{\url{#1}}}
\usepackage{tikz-cd}

\theoremstyle{definition}
\newtheorem{thm}{Theorem}[section]

\newtheorem{cor}[thm]{Corollary}
\newtheorem{lem}[thm]{Lemma}
\newtheorem{prop}[thm]{Proposition}

\newtheorem{defi}[thm]{Definition}

\newtheorem{rek}[thm]{Remark}

\newcommand{\be}{\begin{equation*}}
\newcommand{\ee}{\end{equation*}}
\newcommand{\bel}{\begin{equation}}
\newcommand{\eel}{\end{equation}}
\newcommand{\bea}{\begin{eqnarray}}
\newcommand{\eea}{\end{eqnarray}}
\newcommand{\ben}{\begin{enumerate}}
\newcommand{\een}{\end{enumerate}}
\newcommand{\bi}{\begin{itemize}}
\newcommand{\ei}{\end{itemize}}

\newcommand{\bb}[1]{\mathbb{#1}}
\newcommand{\ul}[1]{\underline{#1}}
\newcommand{\ol}[1]{\overline{#1}}

\newcommand{\ca}[1]{\mathcal{#1}}

\newcommand{\Sym}{\text{Sym}}

\newcommand{\ten}{\otimes}
\newcommand{\bu}{\bullet}

\DeclareMathOperator{\Lie}{Lie}

\DeclareMathOperator{\Spec}{Spec}

\DeclareMathOperator{\dR}{dR}

\DeclareMathOperator{\et}{\acute{e}t}

\DeclareMathOperator{\stab}{stab}

\newcommand{\C}{\ensuremath{\mathbb{C}}}
\newcommand{\Z}{\ensuremath{\mathbb{Z}}}

\newcommand{\F}{\mathbb{F}}

\newcommand{\G}{\mathbb{G}}

\numberwithin{equation}{section}

\usepackage{lastpage}
\usepackage{fancyhdr}
\begin{document}

\title{Arbitrarily large jumps in the de Rham and Hodge cohomology of families in characteristic $p$}

\author{Casimir Kothari}
\address{Department of Mathematics, University of Chicago, Chicago, IL 60637}

\email{\textcolor{blue}{\href{mailto:ckothari@uchicago.edu}{ckothari@uchicago.edu}}}

\date{December 22, 2024}   

\begin{abstract} 
We construct smooth projective families of algebraic varieties in characteristic $p$ such that the dimensions of the de Rham and Hodge cohomology groups of the fibers can be made to jump by an arbitrarily large amount.  To do this, we first construct an example using the classifying stack of a finite flat group scheme which degenerates $\Z/p^2\Z$ to $\alpha_p \oplus \alpha_p$.  Along the way, we give a self-contained exposition of the construction of Godeaux--Serre varieties.
\end{abstract}

\maketitle

\section{Introduction}

If $f: X \to Y$ is a smooth projective family of complex algebraic varieties with $Y$ connected, it is a classical fact that the function on $y \in Y$, 
\be
y \mapsto \dim_{\C} H^i_{\dR}(X_y; \C),
\ee
is constant for any $i \geq 0$.  However, it is well-known that this fails for smooth projective families in mixed and positive characteristic.  In \cite{S}, Suh constructs families of surfaces over a mixed characteristic DVR with an arbitrarily large jump in geometric genus from the generic to the special fiber, while in \cite{CZ}, Cotner and Zavyalov construct a family of surfaces in equal characteristic $p > 0$ such that $\dim H^1_{\dR}(X_s/k(s)) = 2$ and $\dim H^1_{\dR}(X_\eta/k(\eta)) = 1$.  In both cases, the authors use the method of Godeaux--Serre, in which one obtains such families as quotients by the action of certain group schemes on complete intersections in projective space. 

In this article, we use the Godeaux--Serre method to construct families of algebraic varieties in equal characteristic $p>0$ with an arbitrarily large jump in algebraic de Rham and Hodge cohomology (of any degree). For a scheme $X$ over a field $k$, we write $h^i_{\dR}(X/k) := \dim_k H^i_{\dR}(X/k)$ and $h^{i, j}(X/k) := \dim_k H^j(X, \Omega^i_{X/k})$ for the de Rham and Hodge numbers of $X$, respectively. Let $S = \Spec R$ be the spectrum of a DVR of equal characteristic $p$ with closed point $s$ and generic point $\eta$.

\begin{thm} \label{thm:main}
Let $e, n \geq 1$ be positive integers.  Then there exists a smooth projective scheme $X/S$ of relative dimension $n + 1$ with geometrically connected fibers such that for any positive integer $k \leq n$ we have
\bel \label{eqn:main1}
h^k_{\dR}(X_s/k(s)) \geq h^k_{\dR}(X_{\eta}/k(\eta)) + e,
\eel
and for any nonnegative $i, j$ with $1 \leq i + j \leq n$ we have
\bel \label{eqn:main2}
h^{i, j}(X_s/k(s)) \geq h^{i, j}(X_{\eta}/k(\eta)) + e. 
\eel
\end{thm} 
In our examples, the Hodge to de Rham spectral sequence fails to degenerate,\footnote{This follows from our construction and the corresponding fact for the stack $B\alpha_p$ \cite[Proposition 4.12]{ABM}.} so that (\ref{eqn:main2}) does not immediately imply (\ref{eqn:main1}). 
 As a result, we make separate calculations for both de Rham and Hodge cohomology.

\begin{rek}
Base changing along the map $\F_p[t]_{(t)} \to R$ which sends $t$ to the uniformizer of $R$ and noting that algebraic de Rham and Hodge cohomology are compatible with field extensions, we find that it is enough to prove Theorem \ref{thm:main} over $\F_p[t]_{(t)}$.  Thus in the sequel we work over the base $S = \Spec \F_p[t]_{(t)}$.
\end{rek}

\begin{rek}
One can already obtain arbitrarily large de Rham jumps by taking self-products of the surface constructed in \cite{CZ}, however this requires the dimension of the family to increase with the prescribed gap $e$ (though see Remark \ref{rek:alt} for an approach to arbitrary jumps using the \cite{CZ} surface). Therefore the content of the theorem is that we can obtain arbitrarily large jumps in a given cohomological degree in families of fixed dimension.
\end{rek}

The strategy of proof is not new, and can be summarized in the following steps.  First, construct a finite flat group scheme $H/S$ such that $h^k_{\dR}(BH_s) \geq h^k_{\dR}(BH_\eta) + e$ and $h^{i,j}(BH_s) \geq h^{i, j}(BH_\eta) + e$. Second, construct actions of $H$ on projective space such that the complement of the free locus can be made to have arbitrarily large codimension. Finally, choose a complete intersection $Y \subset \bb{P}^N_S$ of relative dimension $n+1$ on which $H$ acts freely, and use the Lefschetz hyperplane theorem to deduce that $Y/H$ and $[\bb{P}^N/H]$ have the same de Rham and Hodge cohomology in a range of degrees.

The structure of this paper is as follows.  In Section \ref{sec:gp} we construct a desirable finite flat $G/S$ and study the cohomology of the classifying stacks of its fibers.  In fact, the group $H$ above will be a power of $G$ depending on the gap $e$. In Section \ref{sec:act} we give an exposition of the Godeaux--Serre method, explaining the construction of complete intersections with free actions of finite flat group schemes.  This section may be read independently of the rest of the paper. In Section \ref{sec:pf} we combine these to prove Theorem \ref{thm:main}.

\section*{Acknowledgements}
I would like to thank my advisor Matthew Emerton for his guidance and encouragement.  Thanks also to Sean Cotner, Luc Illusie, Ray Li, Akhil Mathew, Callum Sutton, Bogdan Zavyalov, and the anonymous referees for helpful conversations and comments relating to this work.  I would especially like to thank Sean Cotner for allowing me to include his construction (Remark \ref{rek:sean}) in this article. This material is based upon work supported by the National Science Foundation Graduate Research Fellowship under Grant No. 2140001.

\section{A group scheme in characteristic $p$} \label{sec:gp}

In this section, we construct a group scheme $G/S$ which will have desirable cohomological properties by modifying a construction of Reid \cite[Section 2.1]{R} of
a group scheme which degenerates $\G_m$ to $\G_a$.  Recall that $S = \Spec \F_p[t]_{(t)}$.

\begin{lem} \label{lem:gp}
There exists a finite flat commutative group scheme $G/S$ of order $p^2$ with $G_\eta \cong \Z/p^2\Z$ and $G_s \cong \alpha_p \oplus \alpha_p$.
\end{lem}

\begin{proof}
We begin by constructing the Cartier dual $G^\vee$ of $G$. First, let $R = \F_p[t]_{(t)}$ and consider the $R$-algebra
\be
A := \frac{R[x, y]}{(x^p, y^p - tx)}.
\ee
$A$ is freely generated as an $R$-module by monomials $x^a y^b$ with $0 \leq a, b \leq p-1$.  We endow $A$ with a Hopf algebra structure $\Delta: A \to A \otimes A, e: A \to R, i: A \to A$ given by
\begin{align*}
&\Delta(y) = 1 \ten y + y \ten 1 + t y \ten y \\
&e(y) = 0 \\
&i(y) = \frac{-y}{1 + ty},
\end{align*}
which by the relation $tx = y^p$ forces $\Delta(x) = 1 \ten x + x \ten 1 + t^{p+1} x \ten x, e(x) = 0$, and $i(x) = \frac{-x}{1 + t^{p+1} x}$.
One checks by hand that axioms for a commutative Hopf algebra hold (note that one only has to check these relations on the generators).

When $t = 0$, we get $k[x, y]/(x^p, y^p)$, with Hopf algebra structure
\begin{align*}
&\Delta(x) = 1 \ten x + x \ten 1, \Delta(y) = 1 \ten y + y \ten 1 \\
&e(x) = e(y) = 0 \\
&i(x) = -x, i(y) = -y,
\end{align*}
which we recognize as $\alpha_p \oplus \alpha_p$, while when $t$ is invertible we get $k(t)[y]/((1 + ~ty)^{p^2} - ~1)$ with Hopf algebra structure
\begin{align*}
&\Delta(1 + ty) = (1 + ty) \otimes (1 + ty)\\
&e(1 + ty) = 1 \\
&i(1 + ty) = \frac{1}{1 + ty},
\end{align*}
which we recognize as $\mu_{p^2}$. Taking the Cartier dual gives the desired $G/S$.
\end{proof}

In the above proof, note that $A/(x) \cong R[y]/(y^p)$ with its given Hopf algebra structure recovers the deformation of $\mu_p$ to $\alpha_p$ constructed in \cite[Proposition 3.1]{R}. It is also interesting to note that the group scheme $G/S$ constructed above is a truncated Barsotti-Tate group generically but is far from being truncated Barsotti-Tate on the special fiber, making it an equal characteristic $p$ example of Raynaud's "affaissement" or "drooping" of truncated Barsotti-Tate groups \cite[Section 3]{Ray2}.

\begin{rek} \label{rek:sean}
An alternate geometric construction of a deformation of $\Z/p^2\Z$ to $\alpha_p \oplus \alpha_p$ due to Sean Cotner is as follows: begin with a degeneration $\ca{E}$ of an ordinary elliptic curve into a supersingular one such that $\ca{E}$ has rational $p^2$-torsion generically.  Then the pullback of the short exact sequence
\be
0 \to \ker V \to \ker V^2 \to \ker V \to 0
\ee
by the Frobenius map $F: \ker V \to \ker V$ yields such a deformation.
\end{rek}

Now we calculate the de Rham and Hodge cohomology of the classifying stacks $BG_s$ and $BG_{\eta}$.  For definitions and basic properties of de Rham and Hodge cohomology of stacks, we refer the reader to \cite[Section 2]{ABM} and \cite[Section 2.4]{CZ}. However, for the calculations in this article, one only needs to know Totaro's theorem (\cite[Theorem 3.1]{T}) and the fact that de Rham and Hodge cohomology satisfy versions of the Lefschetz hyperplane theorem and projective bundle formula in this setting, as in \cite[Section 5]{ABM}.

\begin{prop} \label{prop:coh}
For the group scheme $G/S$ of Lemma \ref{lem:gp}, we have
\be
\lim_{m \to \infty} (h^k_{\dR}(B(G^m_s)) - h^k_{\dR}(B(G^m_{\eta}))) = \infty
\ee
for any $k \geq 1$ and
\be
\lim_{m \to \infty} (h^{i,j}(B(G^m_s)) - h^{i,j}(B(G^m_{\eta}))) = \infty
\ee
for any $i, j \geq 0$ with $i + j \geq 1$. Moreover, the above differences are all nonnegative for any $m \geq 0$. 
\end{prop}

\begin{proof}
First we calculate the Hodge and de Rham cohomology of $B(\Z/p^2\Z)$ and $B\alpha_p$. Since $\Z/p^2\Z$ is a discrete group, one has by \cite[Lemma 10.2]{T} that
\begin{align*}
h^{i,j}(B(\Z/p^2\Z))& = \begin{cases} 1 & i = 0 \\ 0 & i > 0, \end{cases} \\
h^k_{\dR}(B(\Z/p^2\Z))& = 1 
\end{align*}
for all $i,j, k \geq 0$.  Next, by \cite[Proposition 4.10]{ABM} and \cite[Proposition 4.12]{ABM}, we have
\begin{align*}
h^{i,j}(B \alpha_p) &= \begin{cases} 0 & j < i-1 \\ 1 & j = i-1 \text{ or } i = 0 \\ 2 & j \geq i > 0 \end{cases} \\
h^k_{\dR}(B \alpha_p) &= 1 
\end{align*}
for all $i, j, k \geq 0$.  The K\"unneth formula for de Rham cohomology then gives
\be
h^k_{\dR}(B(G^m_\eta))) = \binom{m+k-1}{k}, h^k_{\dR}(B(G^m_s)) = \binom{2m + k - 1}{k}
\ee
from which the first statement of the proposition follows.  By the K\"unneth formula for Hodge cohomology (\cite[Proposition 5.1]{T}), we have
\be
h^{i,j}(B(G_\eta^m)) = \begin{cases} \binom{m+j-1}{j} & i = 0 \\ 0 & i > 0. \end{cases}
\ee
On the special fiber, for $i = 0$ we have
\be
h^{0, j}(B(G_s^m)) = \binom{2m+j-1}{j}.
\ee
For the case of $i > 0$, we argue as follows.  Take $m$ sufficiently large so that we may    choose a fixed partition $i = u_1 + \dots + u_{2m}$ where the $u_i$ are all equal to 0 or 1. The K\"unneth formula and the fact that $h^{i, j}(B\alpha_p) \geq 1$ for $i \in \{0, 1\}$ then yield
\be
h^{i, j}(B(G_s^m)) \geq \sum_{v_1 + \dots + v_{2m} = j} \left(\prod_t h^{u_t, v_t}(B\alpha_p)\right) \geq \sum_{v_1 + \dots + v_{2m} = j} 1 = \binom{2m + j - 1}{j},
\ee
giving the second statement of the proposition. The nonnegativity statement of the proposition is clear from the above calculations.
\end{proof}

\section{The Construction of Godeaux--Serre Varieties} \label{sec:act}

In this section, we give an expository account of the construction of Godeaux--Serre varieties over a noetherian local ring, culminating in Corollary \ref{cor:GS}. Godeaux--Serre varieties are smooth projective schemes which arise as the quotient of a complete intersection by the free action of a finite flat group scheme, and they have proven useful in the construction of interesting examples and pathologies in characteristic $p$ and mixed characteristic algebraic geometry.  One key reason for this utility is the fact that cohomological invariants of Godeaux--Serre varieties are often amenable to calculation via group cohomology of the associated group scheme. For example, Serre used this method to give an example of a smooth projective surface in characteristic $p$ for which Hodge symmetry fails \cite[Proposition 16]{Ser}, and Raynaud used it to give examples of surfaces over a mixed characteristic DVR with non-flat torsion component of the Picard scheme \cite[Proposition 4.2.4]{Ray}.

\subsection{Generalities on Group Actions}
Let $S$ be a locally noetherian scheme, $G$ a finite flat $S$-group scheme, and $X$ an $S$-scheme.  An action of $G$ on $X$ is an $S$-map $G \times_S X \to X$ inducing a group action on $T$-valued points for all $S$-schemes $T$.

\begin{defi}
If $T$ is an $S$-scheme and $x \in X(T)$, the stabilizer of $x$ in $G$ is the functor which sends a $T$-scheme $U$ to $\{g \in G(U) : g x_U = x_U\}$.  The free locus is the functor sending an $S$-scheme $T$ to those points in $X(T)$ with trivial stabilizer.
\end{defi}

\begin{lem} \label{lem:rep}
If $X/S$ is separated, the free locus is represented by an open subscheme of $X$, and is compatible with arbitrary base change. In particular, the free locus is determined by its values on affine $S$-schemes.
\end{lem}

\begin{proof}
See {\cite[Lemma 2.1]{CZ2}}.
\end{proof}

If $X/S$ is separated, we will refer to the complement of the free locus as the \textit{fixed locus}, which by Lemma \ref{lem:rep} is closed in $X$.  In this situation, we say that the action of $G$ on $X$ is free if the free locus is $X$, i.e. if for all $T$ and all $x \in X(T)$, we have that the stabilizer of $x$ in $G$ is trivial.  It is straightforward to check from the definitions that this is equivalent to the usual condition that the graph morphism $G \times_S X \to X \times_S X$ is a monomorphism of schemes.

The following theorem summarizes the facts we will need concerning quotient schemes. For a full treatment of quotients of schemes by the actions of finite flat group schemes, we refer the reader to \cite[Exp. V]{SGA3} and also to the readable treatment in \cite[Ch. 4]{EGM}.

\begin{thm}[Existence and Properties of Quotients] \label{thm:quot}
Let $X$ be a quasiprojective $S$-scheme with an action of a finite flat $S$-group scheme $G$.  Then 
\ben
\item[(a)] A categorical quotient $X/G$ exists as a quasiprojective $S$-scheme and the natural quotient map $\pi: X \to X/G$ is finite and surjective;

\item[(b)] If $S$ is quasi-compact and quasi-separated and $X$ is projective over $S$, then $X/G$ is projective over $S$;

\item[(c)] If the action of $G$ on $X$ is free and $X$ is smooth over $S$, then so is $X/G$;
\een
\end{thm}

\begin{proof}
First, the action $G \times_S X \to X$ yields a groupoid object in schemes, as in \cite[\href{https://stacks.math.columbia.edu/tag/0234}{Tag 0234}]{Sta}. For part (a), the representability of the ringed space quotient $X/G$ is then \cite[Exp. V, Th\'eorème 4.1]{SGA3}, while the quasiprojectivity of the quotient is proven in \cite[Exp. V, Remarque 5.1]{SGA3}.  

For (b), $X/G$ is proper over $S$ by part (a) and \cite[\href{https://stacks.math.columbia.edu/tag/03GN}{Tag 03GN}]{Sta}. Then since $X/G$ is quasiprojective and proper over $S$, it is projective by \cite[\href{https://stacks.math.columbia.edu/tag/0BCL}{Tag 0BCL}]{Sta}.  

Finally we have by \cite[Exp. V, Th\'eorème 4.1]{SGA3} that in the case of a free action the quotient map $X \to X/G$ is finite locally free, so (c) follows by descent.
\end{proof}

\subsection{Complete Intersections with Free Group Actions}

For the purposes of finding projective varieties with free actions, we are interested in group scheme actions on projective spaces. The following arguments are classical (see e.g. the proof of \cite[Lemma 2.1.1]{S}), but we include them for completeness. 

Throughout this subsection, we let $S$ be a noetherian local scheme with  closed point $s$ and residue field $k$.

\begin{prop} \label{prop:ci}
Let $G/S$ be a finite flat group scheme which acts on $\bb{P}^N_S$, let $Z$ be the fixed locus, and suppose that the special fiber $Z_s$ is of dimension $N - r$.  Then for any $1 \leq r' \leq r - 1$, there is a complete intersection $Y \subset \bb{P}^N_S$ of relative dimension $r'$ on which $G$ acts freely, such that the quotient $Y/G$ is a smooth projective scheme over $S$ with geometrically connected fibers. 
\end{prop}

\begin{proof}
Since $\bb{P}^N_S$ is projective, the categorical quotient $\bb{P}^N_S/G$ exists as a projective variety $P$ by Theorem \ref{thm:quot} (b). Let $\pi: \bb{P}^N_S \to P = \bb{P}^N/G$ and $i: P \to \bb{P}^M$ be the quotient map and a closed immersion, respectively.  If $U \subset \bb{P}^N_S$ denotes the free locus, then by Theorem \ref{thm:quot} (c), $U/G$ is smooth over $S$.

By applying Gabber's Bertini theorem \cite[Corollary 1.7]{Gab} in combination with Theorem \ref{thm:quot} (c) on the special fiber, we find that there exists a complete intersection $\ol{L} = V(\ol{f}_1, \dots, \ol{f}_{N-r'}) \subset \bb{P}^M_s$ such that $\ol{L}$ does not intersect $i(\pi(Z_s))$ and $\ol{L} \cap i(P_s)$ is smooth projective over $k$ of dimension $r'$.  In particular, locally on the special fiber of $i(U/G)$, $\ol{f}_1, \dots, \ol{f}_{N-r'}$ form a regular sequence.

Choose lifts $f_i$ of the equations defining $\ol{L}$ to $S$, let $L \subset \bb{P}^M_S$ be the corresponding lift of $\ol{L}$, and set $X := L \cap i(P)$, a projective $S$-scheme.  We claim that $X$ is smooth over $S$. Indeed, first note that $X$ doesn't intersect $i(\pi(Z))$ since there is no intersection on the special fiber and this intersection is proper; thus $X = L \cap i(U/G)$. Since the $\ol{f}_i$ form a regular sequence locally on the special fiber of $i(U/G)$ and $i(U/G)$ is flat over $S$, it follows from \cite[\href{https://stacks.math.columbia.edu/tag/0470}{Tag 0470}]{Sta} that $X$ is flat over $S$. We conclude that all fibers of $X$ are smooth of relative dimension $r$ as the special fiber is smooth and the smooth locus is open on $S$ by \cite[Theorem 12.2.4 (iii)]{EGAIV}.  By the fiberwise criterion for smoothness \cite[\href{https://stacks.math.columbia.edu/tag/01V8}{Tag 01V8}]{Sta}, $X$ is smooth over $S$.

Let $g_i = (i \pi)^*(f_i)$ and $Y = V(g_1, \dots, g_{N-r'}) \subset \bb{P}^N_S$.  Then by construction, $Y$ is a complete intersection of relative dimension $r'$ on which $G$ acts freely with quotient $X$.  Since $Y$ has geometrically connected fibers (being a complete intersection), so does $X$, and the proposition is proven. 
\end{proof}

In order to apply the proposition to arbitrary cohomological jumps, we would like to find actions of our group scheme of interest on projective space such that the codimension of the fixed locus can be made arbitrarily large.  One source of such actions is the regular representation of a group scheme $G = \Spec A$ over $S = \Spec R$ on its ring of functions, whose definition we now recall.

First, to any $R$-module $M$, we may associate the functor $\ul{M}$ on $R$-algebras defined by $\ul{M}(R') = M \ten_R R'$.  In this case, a representation of $G$ on $M$ is defined to be a linear action of $G$ on the functor $\ul{M}$. Moreover, an action of $G$ on an affine scheme $\Spec B$ induces a representation of $G$ on $B$ via left translation of regular functions (see \cite[Part I.2]{Jan} for more details).

Taking the action to be left multiplication $G \times G \to G$, we obtain the regular representation of $G$ on its ring of functions $A$.  Note that since $G$ is a finite flat group scheme and $S$ is local, the functor $\ul{A}$ is represented by the scheme $\Spec(\Sym^\bu_R A^\vee)$, which is a relative affine space over $S$.  We denote this affine $G$-space by $X_G$ below.

One of the goals of this section is to give a detailed proof of the following result of Raynaud, which says that regular representations provide a good source of group actions whose fixed locus has high codimension.

\begin{prop}[{\cite[Lemme 4.2.2 and preceding paragraph]{Ray}}] \label{prop:fixed}
Let $G/S$ be a finite flat group scheme of order $d > 1$, and let $Y \subset \bb{P}(X_G^r)$ be the fixed locus for the action of $G$ by left translation. Then for each $t \in S$, $Y_t$ has codimension at least $r$ in its fiber $\bb{P}(X_G^r)_t$. 
\end{prop}

The following lemma will allow us to reduce the proof of Proposition \ref{prop:fixed} to the case $r = 1$.

\begin{lem} \label{lem:codim}
Let $k$ be an algebraically closed field.  Let $G/k$ be a finite group scheme which acts linearly on an affine space $V$ over $k$ such that the induced action of $G$ on $\bb{P}(V)$ is nontrivial.  Then for any $r \geq 1$, the fixed locus for the induced action of $G$ on $\bb{P}(V^r)$ has codimension at least $r$.
\end{lem}

\begin{proof}
We argue inductively on $r$, where the case of $r = 1$ is automatic as the action of $G$ on $\bb{P}(V)$ is nontrivial. Let $Z_V \subset \bb{P}(V)$ be the fixed locus for the action of $G$, a proper closed subset by assumption.  For $1 \leq i \leq r$, let  $W_i \subset V^r$ be the kernel of the projection $V^r \to V$ onto the $i$-th factor, and set 
\be
Y = \bigcup_{i = 1}^r \bb{P}(W_i) \subset \bb{P}(V^r).
\ee
Then $Y$ is stable under the action of $G$.  We have a projection morphism
\be
\pi: \bb{P}(V^r) \setminus Y \to \bb{P}(V)^r,
\ee
with fibers of dimension $r-1$.  The fixed locus for the action of $G$ on $\bb{P}(V^r) \setminus Y$ is contained in $\pi^{-1}(Z_V^r)$, a space of dimension $r \dim(Z_V) + r - 1$. Since $\dim(Z_V) < n$, the codimension of this space in $\bb{P}(V^r)$ is at least $r$.  The locus $Y$ is the union of $r$ $G$-stable spaces isomorphic (as $G$-spaces) to $\bb{P}(V^{r-1})$, and by induction the fixed locus for the action of $G$ on $\bb{P}(V^{r-1})$ has codimension at least $r-1$, so the result follows.
\end{proof}

\begin{proof}[Proof of Proposition \ref{prop:fixed}]
First, by the compatibility of the free locus with arbitrary base change (Lemma \ref{lem:rep}), $Y_t$ is the fixed locus for the action of $G_t$ on $\bb{P}(X^r_{G_t})$ under the natural identification $\bb{P}(X^r_G)_t \cong \bb{P}(X^r_{G_t})$.  Thus we may pass to geometric fibers and assume that our base is the spectrum of an algebraically closed field $k$, which we take to be of characteristic $p > 0$.\footnote{If the characteristic of $k$ is zero, then the proof is identical except we don't need to consider the non-reduced group schemes $\alpha_p$ and $\mu_p$.}  Then by Lemma \ref{lem:codim}, it suffices to treat the case $r = 1$. 

So let $G/k$ be a nontrivial finite flat group scheme of order $d$. For the regular representation $X := X_G$ of $G$, we consider the subset of lines with nontrivial stabilizer
\be
F := \{[L] \in \bb{P}(X): \stab_G(L) \neq 0\} \subset \bb{P}(X).
\ee
Our goal is to show that $F$ is a proper subset of $\bb{P}(X)$. To do this, we will show that $\dim(F) < d - 1$.

As a first step, we consider a fixed simple subgroup $H \subset G$, so that by the classification of height 1 group schemes via restricted Lie algebras \cite[Exp. VIIA, Remarque 7.5]{SGA3} $H$ is isomorphic to $\Z/p, \alpha_p$, or $\mu_\ell$ for $\ell$ a prime number. Consider the set 
\be
F_H := \{[L] \in \bb{P}(X): H \subset \stab_G(L)\}
\ee
of lines whose stabilizer contains $H$. Clearly $F$ is the union of $F_H$ over all simple subgroups $H \subset G$. Depending on $H$, we obtain a bound for $\dim(F_H)$ as follows:

Case 1: $H \cong \Z/p$.  An action of $\Z/p$ on a line $L$ is given by a homomorphism $\Z/p \to \G_m$.  Since the only such homomorphism is the trivial one, $L \in F_H$ is fixed pointwise.  Thus $L \subset \ul{A}^H \subset X$. But $A^H$ is the coordinate ring of $G/H$, a group scheme of order $d/p$, and consequently $\ul{A}^H \subset X$ is represented by an affine linear subspace of dimension $d/p$.  It follows that $\dim(F_H) \leq  d/p-1$.

Case 2: $H \cong \alpha_p$.  As above, $\alpha_p$ fixes $L$ pointwise.  By the same argument as in case 1, $\dim(F_H) \leq d/p - 1$.

Case 3: $H \cong \mu_\ell$ for $\ell$ any prime.  Then $H$ acts on $L$ through the $i$-th power character $\chi_i: H \to \G_m$ for some $i \in \Z/\ell\Z$. Since $\mu_\ell$ is diagonalizable, the action of $H$ on $X$ induces a $\Z/\ell\Z$-grading $A \cong \oplus_{i \in \Z/\ell \Z} A^i$, such that $H$ acts on $\ul{A}^i$ via $\chi_i$.  But by the construction of the regular representation, the action $H \times X \to X$ is free, so that $A^0 \cong A^H$ and each $A^i$ an invertible $A^0$-module (\cite[Exp. VIII, Proposition 4.1, Th\'eorème 5.1]{SGA3}). We conclude that
\be
L \subset \bigsqcup_{i \in \Z/\ell \Z} \ul{A}^i,
\ee 
and thus $\dim(F_H) \leq \dim(\bb{P}(\ul{A}^0)) = d/\ell - 1$.

Next, to handle the union over all simple subgroups $H$, we consider the cases of \'etale and connected $H$ separately.

First note that $G$ has only finitely many \'etale subgroup schemes and therefore the locus $F_{\et} \subset F$ of lines whose stabilizer contains a nontrivial \'etale subgroup is a union of $F_H$ for finitely many simple $H \subset G$.  By the above calculations for $H = \Z/p$ and $\mu_\ell$ for $\ell \neq p$, we find $\dim(F_H) \leq d/2 - 1$ for any $H$ and thus $\dim(F_{\et}) \leq d/2 - 1$. 

By the classification of height 1 group schemes via restricted Lie algebras, subgroups $H \subset G$ isomorphic to $\alpha_p$ (resp. $\mu_p$) correspond to one-dimensional subspaces of $\Lie(G)$ for which $F = 0$ (resp. $F(x) = x^p$).  We then consider the incidence correspondence
\be
Y = \{(H, L): H \subset G, H \cong \alpha_p \text{ or } \mu_p, L \in F_H\} \subset \bb{P}(\Lie(G)) \times \bb{P}(X).
\ee
By cases 2 and 3 above, any fiber of the projection $Y \to \bb{P}(\Lie(G))$ has dimension at most $d/p - 1$, and therefore $\dim(Y) \leq d/p - 1 + m - 1$, where $m = \dim(\Lie(G))$.  Since the image of $Y$ in $\bb{P}(X)$ is the locus $F_c \subset F$ of lines whose stabilizer contains a nontrivial connected subgroup, we conclude that $\dim(F_c) \leq d/p - 1 + m - 1$.

Combining the above, we find
\be
\dim(F) = \dim(F_{\et} \cup F_c) \leq \max(d/2-1, d/p-1 + m - 1).
\ee
Clearly $d/2 - 1 < d-1$.  Moreover since $p^m$ divides $d$, we have $d/p-1 + m - 1 < d-1$ as well, so that $\dim(F) < d-1$ as desired.
\end{proof}

\begin{cor}[Existence of Godeaux--Serre varieties] \label{cor:GS}
Let $S$ be the spectrum of a noetherian local ring and $G$ a nontrivial finite flat $S$-group scheme.  Then for any $n > 0$, there is an integer $N > 0$ and a complete intersection $Y \subset \bb{P}^N_S$ of relative dimension $n$ on which $G$ acts freely, such that the quotient $Y/G$ is a smooth projective scheme over $S$ with geometrically connected fibers.
\end{cor}

\begin{proof}
This is a direct consequence of Proposition \ref{prop:ci} and Proposition \ref{prop:fixed}.
\end{proof}

\section{Proof of Theorem \ref{thm:main}} \label{sec:pf}
Let $e, n \geq 1$ be positive integers, and let $G/S$ be the group scheme constructed in Lemma \ref{lem:gp}, so that $G_s \cong \alpha_p \oplus \alpha_p$ and $G_\eta \cong \Z/p^2\Z$. By Proposition \ref{prop:coh},  we may choose an $m > 0$ such that for all $1 \leq k \leq n$ and nonnegative $i, j$ with  $1 \leq i + j \leq n$, each of the finitely many inequalities
\begin{align} \label{eqn:ineqdR}
h^k_{\dR}(B(G^m_s)) &\geq h^k_{\dR}(B(G^m_\eta)) + e \\
\label{eqn:ineqHodge} h^{i, j}(B(G^m_s)) &\geq h^{i, j}(B(G^m_\eta)) + e
\end{align}
holds. By Corollary \ref{cor:GS}, there is a projective space $\bb{P}^N_S$ with an action of $G^m$ and a complete intersection $Y \subset \bb{P}^N_S$ of relative dimension $n+1$ on which the action is free, such that $X := Y/G^m$ is a smooth projective scheme over $S$ with geometrically connected fibers.  By the equivariant Lefschetz hyperplane theorem for de Rham and Hodge cohomology    \cite[Proposition 5.10 (2)]{ABM}, we have for each $t \in S$ and all $0 \leq k \leq n$ and $0 \leq i + j \leq n$ that
\begin{align} \label{eqn:lefdR}
H^k_{\dR}(X_t/k(t)) &\cong H^k_{\dR}([\bb{P}^N_{k(t)}/G^m_t]) \\
\label{eqn:lefHodge} H^{j}(X_t, \Omega^i) &\cong H^j([\bb{P}^N_{k(t)}/G^m_t], \Omega^i) 
\end{align}
where $[\bb{P}^N_{k(t)}/G^m_t]$ denotes the quotient stack. But by the construction of Proposition \ref{prop:fixed}, $[\bb{P}^N_{k(t)}/G^m_t]$ is the projectivization of a vector bundle over $BG^m_t$, and therefore by the projective bundle formulae\footnote{The projective bundle formula for Hodge cohomology of syntomic stacks is \cite[Proposition 5.11]{ABM}. The projective bundle formula for the de Rham cohomology of smooth stacks follows from the classical version for schemes \cite[\href{https://stacks.math.columbia.edu/tag/0FMS}{Tag 0FMS}]{Sta} as in \cite[Proposition 5.11]{ABM}.} we  can compute
\begin{align*}
H^k_{\dR}([\bb{P}^N_{k(t)}/G^m_t]) &\cong H^k_{\dR}(\bb{P}^N_{k(t)} \times BG^m_t) \cong H^k_{\dR}(BG^m_t) \oplus H^{k-2}_{\dR}(BG^m_t) \oplus \dots \\
H^j([\bb{P}^N_{k(t)}/G^m_t], \Omega^i) &\cong H^j(\bb{P}^N_{k(t)} \times BG^m_t, \Omega^i) \cong H^j(BG^m_t, \Omega^i) \oplus H^{j-1}(BG^m_t, \Omega^{i-1}) \oplus \dots.
\end{align*}
Combining these with (\ref{eqn:ineqdR}),(\ref{eqn:lefdR}), (\ref{eqn:ineqHodge}), (\ref{eqn:lefHodge}), and the nonnegativity statement of Proposition \ref{prop:coh} we find
\be
h^k_{\dR}(X_s) \geq h^k_{\dR}(X_\eta) + e
\ee
for all $1 \leq k \leq n$ and
\be
h^{i, j}(X_s) \geq h^{i, j}(X_\eta) + e
\ee
for all $i, j$ with $1 \leq i + j \leq n$, as desired.

\begin{rek}\label{rek:alt}
As pointed out to the author by Sean Cotner, one can also obtain arbitrarily large jumps in de Rham cohomology by taking self products of the surface constructed in \cite{CZ} and slicing the result by hyperplanes.  Our approach was motivated by the search for families of group schemes in equal characteristic $p$ that have visibly different cohomological behavior in all degrees on the generic and special fibers (e.g. a non-split extension generically and a split extension on the special fiber).
\end{rek}

\bigskip
\end{document}